\documentclass[11pt,amssymb]{amsart}
\usepackage{amssymb, url, hyperref}
\usepackage[mathscr]{eucal}
\usepackage[all,cmtip]{xy}
\usepackage{amscd}

\newcommand{\im}{{\rm im}\:}

\newcommand{\Ga}{\mathrm{Gal}}
\newtheorem{thm}{Theorem}[section]
\newtheorem{lemma}[thm]{Lemma}
\newtheorem{prop}[thm]{Proposition}

\theoremstyle{definition}

.240pk scaled 1200 .240pk
\usepackage[all,cmtip]{xy}
\usepackage[T2A,T1]{fontenc}
\DeclareSymbolFont{cyrillic}{T2A}{cmr}{m}{n}
\DeclareMathSymbol{\Sha}{\mathalpha}{cyrillic}{216}

\begin{document}

\title[The finiteness of the Tate-Shafarevich group]{\bf The finiteness of the Tate-Shafarevich group over  function fields for algebraic tori defined over the base field}

\begin{abstract}
Let $K$ be a field and $V$ be a set of rank one valuations of $K$. The corresponding Tate-Shafarevich group of a $K$-torus $T$ is $\Sha(T , V) = \ker\left(H^1(K , T) \to \prod_{v \in V} H^1(K_v , T)\right)$. We prove that if $K = k(X)$ is the function field of a smooth geometrically integral quasi-projective variety over a field $k$ of characteristic 0 and $V$ is the set of discrete valuations of $K$ associated with prime divisors on $X$, then for any torus $T$ defined over the base field $k$, the group $\Sha(T , V)$ is finite in the following situations: (1) $k$ is finitely generated and $X(k) \neq \emptyset$; (2) $k$ is a number field.
\end{abstract}

\author[A.S.~Rapinchuk]{Andrei S. Rapinchuk}

\author[I.A.~Rapinchuk]{Igor A. Rapinchuk}

\address{Department of Mathematics, University of Virginia,
Charlottesville, VA 22904-4137, USA}

\email{asr3x@virginia.edu}

\address{Department of Mathematics, Michigan State University, East Lansing, MI 48824, USA}

\email{rapinchu@msu.edu}

\maketitle

\section{Introduction}\label{S-Intro}

Let $K$ be a field equipped with a set $V$ of rank one valuations. Given a $K$-torus $T$, one defines the corresponding Tate-Shafarevich group as follows:
$$
\Sha(T , V) := \ker\left(H^1(K , T) \to \prod_{v \in V} H^1(K_v , T) \right).
$$
A classical result states that if $K$ is a global field and $V$ consists of all valuations of $K$, then $\Sha(T , V)$ is finite. The standard proofs available in the literature deduce this fact from the Nakayama-Tate duality theorem from global class field theory (cf., for example, \cite[Proposition 6.9]{Pl-R} or \cite[\S 11.3]{Voskr}). However, it was pointed out in \cite[Remark 4.5]{RR2} that the finiteness can be easily derived from the finiteness of the class number and the finite generation of the unit group of $K$. In fact, this approach (which we briefly summarize below in \S2) applies in a much more general situation and, in particular, yields the finiteness of $\Sha(T , V)$ for any torus $T$ defined over an arbitrary finitely generated field $K$ when $V$ is a {\it divisorial} set of places, i.e. the set of discrete valuations of $K$ associated with prime divisors on a normal scheme $\mathfrak{X}$ of finite type over $\mathbb{Z}$ with function field $K$ (we will call such an $\mathfrak{X}$ a {\it model} of $K$). Furthermore, in \cite{RR1}, \cite{RR3}, we showed that if $k$ is a field of characteristic zero satisfying Serre's condition (F) (see \cite[Ch. III, \S4.1]{Serre-GC}) and $X$ is a geometrically integral normal $k$-defined variety, then for the function field $K = k(X)$ and the set $V$ of discrete valuations of $K$ corresponding to the prime divisors on $X$ (``{\it geometric} places''), the Tate-Shafarevich group $\Sha(T , V)$ is again finite for any $K$-defined torus $T$.

In a related direction, Harari and Szamuely \cite{HS} observed that if one considers only
tori $T$ defined over the base field $k$, then the finiteness of $\Sha(T , V)$ can be established for the function field $K = k(C)$ of a smooth geometrically integral $k$-defined curve $C$ when

\vskip1mm

\noindent (1) $k$ is finitely generated and $C(k) \neq \emptyset$; or

\vskip1mm

\noindent (2) $k$ is a number field

\vskip1mm

\noindent with respect to the set $V$ of geometric places of $K$ (as opposed to the divisorial sets that can be considered in this case).
The goal of this note is to show that a minor adaptation of the approach developed in \cite{RR1}-\cite{RR3}, in conjunction with some observations made in \cite{HS}, enables one to extend the result of \cite{HS} to function fields of varieties of arbitrary dimension. The precise statement is as follows.
\begin{thm}\label{T:1}
Let $X$ be a smooth geometrically integral variety over a  field $k$ of characteristic 0, and let $V$ be the set of discrete valuations of the function field $K = k(X)$ associated with codimension one points of $X$. Then for any $k$-defined torus $T$, the Tate-Shafarevich group $\Sha(T , V)$ is finite in the following situations:

\vskip2mm

{\rm (1)} $k$ is finitely generated and $X(k) \neq \emptyset$;

\vskip1mm

{\rm (2)} $k$ is a number field.
\end{thm}

\vskip2mm

\noindent {\bf Remark 1.2.}

\vskip1mm

\noindent (i) \parbox[t]{16cm}{If $T$ is not assumed to be defined over the base field $k$, then, to the best of our knowledge, the question about the finiteness of $\Sha(T , V)$ in this situation remains open, even when $X$ is a curve.}

\vskip1mm

\noindent (ii) \parbox[t]{16cm}{For our purposes, we may (and will) assume that the variety $X$ is affine --- this can always be achieved by replacing $X$ with a suitable affine open $k$-subvariety, which will only shrink the corresponding set $V$ of valuations of $K = k(X)$. While the affineness assumption
is not essential for the argument, it will help us to avoid certain technicalities, while still yielding the required finiteness statement. By the same token, we can avoid the smoothness assumption, although then we need to require in item (1) of the theorem that $X$ has a {\it smooth} $k$-rational point.}

\vskip3mm

The paper is organized as follows. In \S\ref{S:Review}, we provide a short summary of our adelic approach to the finiteness of Tate-Shafarevich groups of tori over finitely generated fields with respect to divisorial sets of valuations. We then turn to the proof of Theorem \ref{T:1} in \S\ref{S:Proof}. Finally, in \S\ref{S:Extensions}, we make some brief remarks on higher Tate-Shafarevich groups.

\vskip2mm

\noindent {\bf Notation.} Suppose $T$ is a torus over a field $K.$ For a Galois extension $L/K$, we follow the usual practice (cf. \cite[Ch. III]{Serre-GC}) and denote by $H^i(L/K, T)$ the Galois cohomology group $H^i(\Ga(L/K), T(L))$; if $L = K^{\rm sep}$ is a separable closure of $K$, we denote the latter group simply by $H^i(K, T).$ For extra clarity, we will occasionally write $H^i(L/K, T(L)).$

\section{A brief review of adeles and the method from \cite{RR1}--\cite{RR3}}\label{S:Review}

We will now quickly recall some of the relevant terminology pertaining to adele groups and summarize several important points of our approach to the finiteness of Tate-Shafarevich groups with respect to divisorial sets, which we will also use below in the function field case --- we refer the reader to \cite[\S\S3-4]{RR2} for the full details.

So, suppose $K$ is a field equipped with a set $V$ of discrete valuations. As usual, one defines the ring of adeles $\mathbb{A}_K(V)$ as the restricted product of the completions $K_v$ for $v \in V$ with respect to the valuation rings $\mathcal{O}_v \subset K_v$. The group of invertible elements of $\mathbb{A}_K(V)$ is the group of ideles $\mathbb{I}_K(V)$. In the sequel, we will assume that $V$ satisfies the following condition (which holds automatically for the divisorial and geometric sets of places that we consider below):

\vskip2mm

\noindent $(*)$ \parbox[t]{12cm}{for any $a \in K^{\times}$, the set $V(a) = \{v \in V \ \vert \ v(a) \neq 0 \}$ is finite.}

\vskip2mm

\noindent We then have diagonal embeddings $K \hookrightarrow \mathbb{A}_K(V)$ and $K^{\times} \hookrightarrow \mathbb{I}_K(V)$. In particular, $\mathbb{A}_K(V)$ has a natural structure of a $K$-algebra.
Next, let
$$
\mathbb{A}_K^{\infty}(V) = \prod_{v \in V} \mathcal{O}_v \ \ \text{and} \ \ \mathbb{I}_K^{\infty}(V) = \prod_{v \in V} \mathcal{O}_v^{\times}
$$
be the subring of integral adeles and the subgroup of integral ideles, respectively. We recall that the quotient $\mathbb{I}_K(V)/ (\mathbb{I}_K^{\infty}(V) \cdot K^{\times})$ can be identified with the Picard group $\mathrm{Pic}(K , V)$, which is defined as the quotient of the group of divisors $\mathrm{Div}(K , V)$, i.e. the free abelian group on $V,$ by the subgroup of principal divisors $\mathrm{P}(K , V)$, where the principal divisor corresponding to $a \in K^{\times}$ is $\sum_{v \in V} v(a) \cdot v$. Furthermore, suppose $L/K$ is a finite Galois extension with Galois group $\mathscr{G} = \mathrm{Gal}(L/K)$, and let $\bar{V}$ be the set of all extensions of valuations from $V$ to $L$. We then have an isomorphism $\mathbb{A}_K(V) \otimes_K L \simeq \mathbb{A}_L(\bar{V})$ of topological rings. This, in particular, enables us to define an action of $\mathscr{G}$ on $\mathbb{A}_L(\bar{V})$ with the property that
$$
\mathbb{A}_L(\bar{V})^{\mathscr{G}} = \mathbb{A}_K(V).
$$

Next, let $T$ be a $K$-torus. Since $\mathbb{A}_K(V)$ is a $K$-algebra, we can consider the adelic group $T(\mathbb{A}_K(V)).$ For each $v \in V$, let $T(\mathcal{O}_v)$ be the unique maximal bounded subgroup of $T(K_v)$. (The notation is justified by the fact that this subgroup is indeed obtained as the group of $\mathcal{O}_v$-points of a suitable $\mathcal{O}_v$-model $\mathcal{T}$ of $T \times_K K_v$.)
One can show that $T(\mathbb{A}_K(V))$ is simply the restricted product of the $T(K_v)$ for $v \in V$ with respect to the subgroups $T(\mathcal{O}_v)$ (see, for example, \cite[\S11.1]{Voskr} for the number field case). We define the subgroup of integral adeles to be
$$
T(\mathbb{A}_K^{\infty}(V)) = \prod_{v \in V} T(\mathcal{O}_v).
$$
Furthermore, the diagonal embedding $K \hookrightarrow \mathbb{A}_K(V)$ yields an embedding $T(K) \hookrightarrow T(\mathbb{A}_K(V))$, which is also often referred to as the diagonal embedding. For a field extension $L/K$, the corresponding groups $T(\mathbb{A}_L(\bar{V}))$ and $T(\mathbb{A}_L^{\infty}(\bar{V}))$ are defined analogously. Moreover, if, as above, $L/K$ is a finite Galois extension with Galois group $\mathscr{G}$, then the action of $\mathscr{G}$ on $\mathbb{A}_L(\bar{V})$ naturally induces a $\mathscr{G}$-action on $T(\mathbb{A}_L(\bar{V})).$ In this case, the diagonal embedding gives rise to the group homomorphism
$$
\lambda_{L/K} \colon H^1(L/K, T) \to H^1(L/K, T(\mathbb{A}_L(\bar{V}))).
$$
Also, we set $E(T, V, L) = T(L) \cap T(\mathbb{A}_L^{\infty}(\bar{V})).$ With these preliminaries, we have the following statement.
\begin{prop}\label{P:1}
Let $T$ be a $K$-torus that splits over a finite Galois extension $L/K$. Then

\vskip2mm

\noindent {\rm (i)}  We have $\Sha(T , V) = \ker \lambda_{L/K}$.

\vskip2mm

\noindent {\rm (ii)} \parbox[t]{15cm}{If $T(\mathbb{A}_L(\bar{V})) = T(\mathbb{A}_L^{\infty}(\bar{V})) \cdot T(L),$ then $\Sha(T , V)$ is contained in
the image of the map $\nu \colon H^1(L/K , E(T, V, L)) \to H^1(L/K , T)$.}
\end{prop}

\begin{proof}
These facts were basically established in \cite[\S 4]{RR2}, so we will just give a sketch of the argument. Let us define
$$
\Sha(L/K, T, V) = \ker \left(H^1(L/K, T) \to \prod_{v \in V} H^1(L_{\bar{v}}/K_v, T) \right),
$$
where, in the product on the right, we choose, for each $v \in V$, a {\it single} extension $\bar{v} \vert v$ in $\bar{V}$. First, it follows from Hilbert's Theorem 90 and the inflation-restriction exact sequence that $\Sha(T, V) = \Sha(L/K, T, V)$. Second, one easily shows (cf. the argument on pp. 244--245 in \cite{RR1}) that if the extension $L_{\bar{v}}/K_v$ is unramified, then the map
$$
\iota_v \colon H^1(L_{\bar{v}}/K_v , T(\mathcal{O}_{L_{\bar{v}}})) \longrightarrow H^1(L_{\bar{v}}/K_v , T(L_{\bar{v}}))
$$
is injective. (We  note that the injectivity of $\iota_v$ can also be seen in the more general framework of results on the Grothendieck-Serre conjecture over discrete valuation rings --- cf. \cite{Guo}, \cite{Nisn2}.) These two facts together, combined with Shapiro's Lemma, yield (i) (cf. {\it loc. cit.} for the details).

Next, suppose that an element $\xi \in H^1(L/K , T)$, represented by a cocycle $\{ \xi_{\sigma}\}$ ($\sigma \in \Ga(L/K)$), lies in $\ker \lambda_{L/K}$. Then there exists $a \in T(\mathbb{A}_L(\bar{V}))$ such that
$$
\xi_{\sigma} = \sigma(a) \cdot a^{-1} \ \ \text{for all} \ \ \sigma \in \Ga(L/K).
$$
The assumption made in (ii) implies that one can write $ a = b \cdot c$, with $b \in T(\mathbb{A}_L^{\infty}(\bar{V}))$ and $c \in T(L)$. Then for any $\sigma \in \Ga(L/K)$, we have
$$
\varepsilon_{\sigma} := \xi_{\sigma} \cdot (\sigma(c) c^{-1})^{-1} = \sigma(b) b^{-1} \in E(T, V, L).
$$
Thus, $\{\epsilon_{\sigma}\}$ is a cocycle on $\Ga(L/K)$ with values in $E(T, V, L)$, and the previous equation shows that for the corresponding cohomology class $\varepsilon \in H^1(L/K, E(T, V, L))$, we have $\xi = \nu(\varepsilon)$, proving (ii). (We note that this argument is a simplified version of the proof of Proposition 4.4 in \cite{RR2} as due to the assumption made in the statement of (ii) above, the group $H$ in {\it loc. cit.} coincides with $T(\mathbb{A}(L , V^L))$ in the notations used there.)

\end{proof}

Suppose now that $K$ is a finitely generated field and $V$ is a divisorial set of places of $K$ associated with a model $\mathfrak{X}$. Let $T$ be a
$K$-torus,  $L/K$ be a finite Galois extension that splits $T$, and $\bar{V}$ be the set of all extensions of valuations from $V$ to $L$. Then the Picard group $\mathrm{Pic}(L , \bar{V})$ coincides with the Picard group of the normalization $\tilde{\mathfrak{X}}$ of $\mathfrak{X}$ in $L$, hence is finitely generated (see \cite[Corollaire 1]{Kahn}). Consequently, there exists a subset $V' \subset V$ with finite complement such that $\mathrm{Pic}(L , \bar{V'})$ is trivial, and therefore $\mathbb{I}_L(\bar{V'}) = \mathbb{I}_L^{\infty}(\bar{V'}) \cdot L^{\times}$ (see, for example, \cite[Proposition 3.1]{RR2}). Since $T$ is split over $L$, the latter implies that $T(\mathbb{A}_L(\bar{V'})) = T(\mathbb{A}_L^{\infty}(\bar{V'})) \cdot T(L)$.
Then according to Proposition \ref{P:1}(ii), the group $\Sha(T , V')$ (hence also $\Sha(T , V)$) is contained in the image of the map $\nu' \colon H^1(L/K , E(T, V', L)) \to H^1(L/K , T)$. But in this situation, we show that the group $E(T,V',L)$ is finitely generated (see the proof of \cite[Proposition 3.2]{RR2}), and therefore $H^1(L/K , E(T, V', L))$ is finite. In summary, we obtain the following finiteness result.

\begin{prop}\label{P:ShaFG} {\rm (\cite[Theorem 1.2]{RR2})}
Let $K$ be a finitely generated field and $V$ be a divisorial set of places of $K$ associated to a model $\mathfrak{X}$. Then for any algebraic $K$-torus $T$, the Tate-Shafarevich group
$$
\Sha(T,V) = \ker \left( H^1(K,T) \to \prod_{v \in V} H^1(K_v, T) \right)
$$
is finite.
\end{prop}

\section{Tate-Shafarevich groups over function fields}\label{S:Proof}

In this section, we will discuss how minor adaptations of the method reviewed above, together with several additional observations (some of which were already present in \cite{HS}) lead to a proof of Theorem \ref{T:1}.

Our set-up is as follows. Suppose $X$ is a smooth geometrically integral affine variety over a field $k$. Let $K = k(X)$ be the corresponding function field and $V$ be the set of discrete valuations of $K$ associated with the prime divisors on $X$. Also, let $T$ be a torus defined over the base field $k$ that splits over a finite Galois extension $\ell/k$ with $\mathscr{G} = \Ga(\ell/k)$. Set $L = \ell \cdot K$ and let $\bar{V}$ be the set of all extensions of the valuations from $V$ to $L$.

Now, let us assume in addition that $k$ is a finitely generated field. Then, as above, $\mathrm{Pic}(L , \bar{V})$ is finitely generated (see \cite[Corollaire 2]{Kahn} or \cite[Proposition 6.1]{GJRW}), so there exists a subset $V' \subset V$ with finite complement such that $\mathbb{I}_L(\bar{V'}) = \mathbb{I}_L^{\infty}(\bar{V'}) \cdot L^{\times}$, and hence also $T(\mathbb{A}_L(\bar{V'})) = T(\mathbb{A}_L^{\infty}(\bar{V'})) \cdot T(L)$. Note that we may assume that $V'$ is the set of discrete valuations of $K$ associated with the prime divisors on an open affine subvariety $X' \subset X$ (see, for example, the discussion in \cite[\S5.3]{RR-Survey}). From
Proposition \ref{P:1}(ii), we conclude that $\Sha(T,V')$ (hence also $\Sha(T,V)$) is contained in the image of the map
$$
\nu' \colon H^1(L/K , E(T, V', L)) \to H^1(L/K , T),
$$
where, as before, $E(T,V', L) = T(L) \cap T(\mathbb{A}_L^{\infty}(\bar{V'})).$
However, in contrast to the case of divisorial sets of valuations reviewed above, the group $E(T, V', L)$ does not have to be finitely generated. Nevertheless, we have the following. For the statement, we note that we have an inclusion $T(\ell) \subset E(T, V', L).$
\begin{lemma}\label{L:FG}
With notations as above, the quotient $E(T, V', L)/T(\ell)$ is a finitely generated abelian group.
\end{lemma}
\begin{proof}
Since, by assumption, $T$ splits over $\ell$, it suffices to consider the case of a 1-dimensional $\ell$-split torus $T = \mathbb{G}_m$. Then $E(T, V', L)$ coincides with
$$
U(L , \bar{V'}) = \{ a \in L^{\times} \ \vert \ v(a) = 0 \ \ \text{for all} \ \ v \in \bar{V'} \}.
$$
Let $\tilde{X'}$ be the normalization of $X'$ in $L$. Since $\tilde{X'}$ is affine, we can embed it as a dense open subset into a projective $\ell$-variety $Y$ (of course, this fits into the general framework of Nagata's compactification theorem --- see \cite{ConNag} and \cite{Lut} for proofs written in scheme-theoretic language). Furthermore, since $\tilde{X'}$ is normal, and it is well-known that the normalization of a projective variety is projective (see, for example, \cite[Tag 0BXQ, Lemma 33.27.2]{Stacks}), we can assume that $Y$ is normal as well. We have $\ell(Y) = \ell(\tilde{X'}) = L$, and we denote by $V''$ the set of discrete valuations of $L$ associated with prime divisors on $Y$. Then $V'' \setminus \bar{V'}$ consists of finitely many valuations, say, $v_1, \ldots , v_r$ (which correspond to the prime divisors on $Y$ not contained in $\tilde{X'}$). Consider the group homomorphism
$$
\tau \colon U(L , \bar{V'}) \to \mathbb{Z}^r, \ \ \ \ \tau(a) = (v_1(a), \cdots, v_r(a)).
$$
Since $Y$ is normal and complete, any function $a \in L$ that satisfies $v(a) \geq 0$ for all $v \in V''$ belongs to $\ell$ (see, for example, \cite[Theorem 6.45]{GW} or \cite[Proposition 6.3A]{Hartshorne}). It follows that
$$\ker \tau = \{ a \in L^{\times} \ \vert \ v(a) = 0 \ \ \text{for all} \ \ v \in V''\}$$ reduces to $\ell^{\times}$. So, being isomorphic to a subgroup of $\mathbb{Z}^r$, the quotient $U(L, \bar{V'})/\ell^{\times}$ is a finitely generated group, as required.
\end{proof}

Thus, we have a short exact sequence of modules over $\mathscr{G} = \mathrm{Gal}(\ell/k) = \mathrm{Gal}(L/K)$:
$$
1 \to T(\ell) \longrightarrow E(T, V', L) \longrightarrow \Gamma \to 1,
$$
where $\Gamma$ is finitely generated as an abelian group. We then have the following exact sequence of cohomology groups
$$
H^1(\ell/k , T) \stackrel{\delta}{\longrightarrow} H^1(L/K , E(T, V', L)) \longrightarrow H^1(L/K , \Gamma).
$$
Since the group $H^1(L/K , \Gamma)$ is finite (see, for example, \cite[Ch. II, Corollary 1.32]{MilneCFT}), we see that the intersection $\Sha(T , V) \cap (\im (\nu' \circ \delta))$ has finite index in $\Sha(T , V)$. We now define
$$
\Sha_0(T , V) = \ker\left(H^1(\ell/k , T) \longrightarrow \prod_{v \in V} H^1(L_{\bar{v}}/K_{v} , T)   \right), 
$$
where, as in the proof of Proposition \ref{P:1}, we fix a {\it single} extension $\bar{v} \vert v$ in $\bar{V}.$
Clearly, we have $\Sha(T , V) \cap (\im (\nu' \circ \delta)) = (\nu' \circ \delta)(\Sha_0(T , V))$. So, we obtain the following.

\begin{prop}\label{P:2}
With notations as above, if $\Sha_0(T , V)$ is finite, then $\Sha(T , V)$ is also finite.
\end{prop}

Thus, in order to prove Theorem \ref{T:1}, it is enough to show in the situations described in the items (1) and (2), the group $\Sha_0(T , V)$ is finite. For this, we will consider two cases.

\vskip4mm

\noindent \underline{{\bf The case $\dim X = 1$.}} This case was studied by Harari and Szamuely in \cite{HS}, and we will now review the relevant parts of their argument to put the general case in the appropriate context.

First, we have the following observation.

\begin{lemma}\label{L:1}
Let $X$ be a smooth geometrically integral curve over an arbitrary field $k$. Suppose $x \in X(k)$ is a $k$-rational point and let $v$ be the discrete valuation of $K = k(X)$ associated with $x$. Then for any $k$-torus $T$ that splits over a finite Galois extension $\ell/k$, the map
$$
\mu_v \colon H^1(\ell/k , T) \longrightarrow H^1(L_{\bar{v}}/K_v , T), \ \ \ \text{where} \ \ L = \ell \cdot K \ \ \text{and} \ \ \bar{v} \vert v,
$$
is injective.
\end{lemma}
\begin{proof}
Let us consider the following diagram
\begin{equation}\label{E:30}
\xymatrix{H^1(\ell/k , T) \ar[r]^(0.40){\mu'_v} \ar@{-->}[rd]_{{\rm id}} &  H^1(L_{\bar{v}}/K_v , T(\mathcal{O}_{L_{\bar{v}}})) \ar[d]^{\rho_v} \ar[r]^{\mu''_v} & H^1(L_{\bar{v}}/K_v , T(L_{\bar{v}})) \\ & H^1(\ell/k, T) &}
\end{equation}
where the top row represents a natural factorization of $\mu_v$, and $\rho_v$ is induced by the reduction map modulo the valuation ideal $\mathfrak{P}_{L_{\bar{v}}}$
of the valuation ring $\mathcal{O}_{L_{\bar{v}}}$ (we note that the residue fields of $L_{\bar{v}}$ and $K_v$ are $\ell$ and $k$, respectively). Since the extension $L_{\bar{v}}/K_v$ is unramified, the map $\mu''_v$ is injective (see the proof of Proposition \ref{P:1}),
so $\ker \mu_v = \ker \mu'_v$. On the other hand, the composition $\rho_v \circ \mu'_v$ is the identity map, and hence $\ker \mu'_v$ is trivial.
\end{proof}

It follows immediately from the lemma that if $X(k) \neq \emptyset$, then $\Sha_0(T , V)$ is trivial. Now suppose that $k$ is a number field and denote by $V^k_f$ the set of all finite places of $k$. By the Lang-Weil estimates \cite{LW} and Hensel's lemma, we can find a subset $U \subset V^k_f$ with finite complement $V^k_f \setminus U$ such that for any $u \in U$, there exists a point $x \in X$ whose residue field $k(x)$ is contained in the completion $k_u$ (cf. \cite[Lemma 2.1]{HS}). Furthermore, it is well-known (and follows, for example, from Proposition \ref{P:ShaFG}) that
the group
$$
\Sha_k(T , U) := \ker\left( H^1(k , T) \longrightarrow \prod_{u \in U} H^1(k_u , T)\right)
$$
is finite. Thus, to prove the finiteness of $\Sha_0(T , V)$, it is enough to show that there is an inclusion
\begin{equation}\label{E:5}
\Sha_0(T , V) \subset \Sha_k(T , U).
\end{equation}
For this, fix $u \in U$, pick $x \in X$ such that $k(x) \subset k_u$, and let $v$ be the discrete valuation of $K = k(X)$ corresponding to $x$. Set $X_u = X \times_k k_u$. Then $x$ lifts to a point $x_u \in X_u(k_u)$, and we denote by $v_u$ the discrete valuation of $K^u = k_u(X_u)$ associated with $x_u$. We then have the following commutative diagram
$$
\xymatrix{H^1(\ell/k , T) \ar[r]^{\mu_v} \ar[d]_{\lambda_u} & H^1(L_{\bar{v}}/K_v , T) \ar[d]^{\Lambda_u} \\ H^1(\ell_{\bar{u}}/k_u , T) \ar[r]^(.40){\mu_{v_u}} & H^1((L^u)_{\bar{v}_u}/(K^u)_{v_u} , T)}
$$
where $L^u = \ell \cdot K^u$ and  $\bar{v}_u \vert v_u$. Suppose now that $\xi \in \Sha_0(T , V)$. Then
$$
\Lambda_u(\mu_v(\xi)) = 1 = \mu_{v_u}(\lambda_u(\xi)).
$$
Since $x_u \in X(k_u)$, the map $\mu_{v_u}$ is injective by Lemma \ref{L:1}, and therefore $\lambda_u(\xi) = 1$.
As $u \in U$ was arbitrary, we obtain (\ref{E:5}), which completes the consideration of the case where $\dim X = 1$.

\vskip4mm

\noindent \underline{{\bf The case $\dim X \geq 2$.}} We begin the argument with the following upgrade of Lemma \ref{L:1}.
\begin{lemma}\label{L:2}
Let $X$ be a smooth geometrically integral variety over an arbitrary field $k$, let $Y \subset X$ be a smooth geometrically integral $k$-defined subvariety of codimension 1, and let $v$ be the discrete valuation of the function field $K = k(X)$ associated with $Y$. If $Y(k) \neq \emptyset$, then for any $k$-torus $T$ that splits over finite Galois extension $\ell/k$, the map
$$
\mu_v \colon H^1(\ell/k , T) \longrightarrow H^1(L_{\bar{v}}/K_v , T),  \ \ \ \text{where} \ \ L = \ell \cdot K \ \ \text{and} \ \ \bar{v} \vert v,
$$
is injective.
\end{lemma}
\begin{proof}
We start with a diagram similar to (\ref{E:30}):
$$
\xymatrix{H^1(\ell/k , T) \ar[r]^(0.40){\mu'_v} \ar@{-->}[rd]_{\nu_v} &  H^1(L_{\bar{v}}/K_v , T(\mathcal{O}_{L_{\bar{v}}})) \ar[d]^{\rho_v} \ar[r]^{\mu''_v} & H^1(L_{\bar{v}}/K_v , T(L_{\bar{v}})) \\ & H^1(L^{(\bar{v})}/K^{(v)}, T) &}
$$
where $L^{(\bar{v})}$ and $K^{(v)}$ denote the corresponding residue fields. Again, $\mu''_v$ is injective, so we conclude that $\ker \mu_v$ is contained in the kernel of $\nu_v = \rho_v \circ \mu'_v$. In the case at hand, we have $K^{(v)} = k(Y)$ and $L^{(\bar{v})} = \ell(Y_{\ell})$, where $Y_{\ell} = Y \times_k \ell$. Thus, we need to prove the injectivity of
$$
\nu_Y \colon H^1(\ell/k , T) \longrightarrow H^1(\ell(Y_{\ell})/k(Y) , T).
$$
Fix $y \in Y(k) \subset Y_{\ell}(\ell)$ and consider the local rings $\mathcal{O}_{Y , y}$ and $\mathcal{O}_{Y_{\ell} , y} = \ell \cdot \mathcal{O}_{Y , y}$, with maximal ideals $\mathfrak{m}$ and $\mathfrak{m}_{\ell}$. The residue fields $\mathcal{O}_{Y , y}/\mathfrak{m}$ and $\mathcal{O}_{Y_{\ell} , y}/\mathfrak{m}_{\ell}$ coincide with $k$ and $\ell$, respectively, and we have another diagram analogous to (\ref{E:30}):
$$
\xymatrix{H^1(\ell/k , T) \ar[r]^(0.38){\nu'_Y} \ar@{-->}[rd]_{{\rm id}} & H^1(\ell(Y_{\ell})/k(Y) , T(\mathcal{O}_{Y_{\ell} , y}))  \ar[d]^{\rho_{Y_{\ell} , y}} \ar[r]^(.59){\nu''_Y} & H^1(\ell(Y_{\ell})/k(Y) , T) \\ & H^1(\ell/k, T) &}
$$
where again the top row is a factorization of $\nu_Y$ and $\rho_{Y_{\ell} , y}$ is induced by the residue map modulo $\mathfrak{m}_{\ell}$. Since $Y$ is smooth, by \cite[Theorem 4.1]{CT-S} (or also by \cite[Theorem 1]{FP}), the map $\nu''_Y$ is injective. But the composition $\rho_{Y_{\ell} , y} \circ \nu'_Y$ is the identity map, so $\nu'_Y$ is injective, and therefore $\nu_Y$ is injective, as required.
\end{proof}

We will use this statement in conjunction with the following consequence of Bertini-type theorems.

\begin{prop}\label{P:3}
Let $X$ be a smooth geometrically integral quasi-projective variety of dimension $\geq 2$ over a field $k$ of characteristic 0. Then $X$ contains a smooth geometrically integral $k$-defined subvariety $Y$ of codimension $1$. Moreover, given $x \in X(k)$, one can choose such a $Y$ so that it contains $x$.
\end{prop}
This follows immediately from \cite[Theorem 3.6]{GK} (see also \cite{KA} for various related results).

\vskip3mm

We can now complete the proof of Theorem \ref{T:1}. First, suppose $X$ contains a $k$-rational point $x$. According to Proposition \ref{P:3}, there exists a geometrically integral smooth $k$-defined subvariety $Y \subset X$ of codimension 1 such that $Y(k) \neq \emptyset$. Let $v$ be the discrete valuation of $K = k(X)$ associated with $Y$. Then according to Lemma \ref{L:2}, the map $\mu_v \colon H^1(\ell/k , T) \longrightarrow H^1(L_{\bar{v}}/K_v , T)$ is injective, hence $\Sha_0(T , V)$ is trivial.

Now assume that $k$ is a number field. By Proposition \ref{P:3}, we can find a geometrically integral smooth $k$-defined subvariety $Y \subset X$ of codimension $1$. Let again $v$ be the discrete valuation of $K = k(X)$ associated with $Y$. It follows from the Lang-Weil estimates (cf. \cite{LW}) and Hensel's lemma that there exists a subset $U \subset V^K_f$ with finite complement $V^K_f \setminus U$ such that $Y(k_u) \neq \emptyset$ for all $u \in U$. We claim that
\begin{equation}\label{E:10}
\Sha_0(T , V) \subset \Sha_k(T , U) := \ker\left( H^1(\ell/k , T) \longrightarrow \prod_{u \in U}  H^1(\ell_{\bar{u}}/k_u , T) \right).
\end{equation}
Since again $\Sha_k(T , U)$ is finite, this will yield the desired finiteness of $\Sha_0(T , V)$.

Fix $u \in U$. Set $X_u = X \times_k k_u$ and $Y_u = Y \times_k k_u$, and let $v_u$ denote the discrete valuation of $K^u := k_u(X_u)$ associated with $Y_u$. We then have the following commutative diagram
$$
\xymatrix{H^1(\ell/k , T) \ar[r]^{\mu_v} \ar[d]_{\lambda_u} & H^1(L_{\bar{v}}/K_v , T) \ar[d]^{\Lambda_u} \\ H^1(\ell_{\bar{u}}/k_u , T) \ar[r]^(.40){\mu_{v_u}} & H^1((L^u)_{\bar{v}_u}/(K^u)_{v_u} , T)}
$$
where $L^u = \ell \cdot K^u$ and  $\bar{v}_u \vert v_u$. Suppose now that $\xi \in \Sha_0(T , V)$. Then
$$
\Lambda_u(\mu_v(\xi)) = 1 = \mu_{v_u}(\lambda_u(\xi)).
$$
Since $Y(k_u) \neq \emptyset$, by Lemma \ref{L:2}, the map $\mu_{v_u}$ is injective, and therefore $\lambda_u(\xi) = 1$. Since $u \in U$ was arbitrary, this establishes the inclusion (\ref{E:10}) and completes the proof of Theorem \ref{T:1}.

\section{Extensions}\label{S:Extensions}

So far, we have worked only with Tate-Shafarevich groups in degree 1. However, one can also consider higher Tate-Shafarevich groups. More precisely, suppose $K$ is a field equipped with a set $V$ of rank one valuations. Then for a $K$-torus $T$, a (finite) Galois extension $L/K$, and any $i \geq 1$ one can consider the $i$th Tate-Shafarevich group
$$
\Sha^i(L/K, T, V) := \ker\left(H^i(L/K , T) \longrightarrow \prod_{v \in V} H^i(L_{\bar{v}}/K_v , T)    \right),
$$
where as before, in the product on the right, we choose a {\it single} extension $\bar{v}$ of $v$ to $L$.

We showed in \cite[Proposition 4.2]{RR2} that if $K$ is a finitely generated field and $V$ is a divisorial set of places, then $\Sha^i(L/K , T, V)$ is finite for any $K$-torus $T$, any finite Galois extension $L/K$, and {\it all} $i \geq 1$. Furthermore, in \cite{IR}, the second-named author modified the techniques used in the proof of this result to establish the finiteness of $\Sha^i(L/K, T, V)$ for all $i \geq 2$ in the case where $K = k(X)$ is the function field of a smooth geometrically integral variety $X$ defined over a field $k$ of type (F) in the sense of Serre, $V$ consists of the discrete valuations of $K$ associated with prime divisors on $X$, and $T$ is an arbitrary $K$-torus (for $i = 1$ this result was obtained earlier in \cite{RR2}, \cite{RR3}).

We will now briefly discuss results of this kind in the set-up considered in Theorem \ref{T:1}. Thus, we take $K = k(X)$ to be a function field of a smooth geometrically integral variety $X$ over a finitely generated field $k$, let $V$ be the set of discrete valuations associated with the divisors on $X$, but now we assume that $T$ is defined over the base field $k$. Let $\ell/k$ be a finite Galois extension and set $L = \ell \cdot K$. First, it follows from \cite[\S 4]{RR2} that the assertion of Proposition \ref{P:1} remains valid for $\Sha^i(L/K, T, V)$ for any $i > 1$, so, in view of Lemma \ref{L:FG}, the question about the finiteness of $\Sha^i(L/K, T, V)$ reduces to the finiteness of
$$
\Sha^i_0(\ell/k, T, V) := \ker\left( H^i(\ell/k , T) \longrightarrow \prod_{v \in V} H^i(L_{\bar{v}}/K_v , T)  \right).
$$

Let us start with the case where $\dim X = 1.$ We note that for any $v \in V$, the extension $L_{\bar{v}}/K_v$ is unramified, and hence the map
$H^i(L_{\bar{v}}/K_v , T(\mathcal{O}_{L_{\bar{v}}})) \to H^i(L_{\bar{v}}/K_v , T(L_{\bar{v}}))$ is injective for all $i \geq 1$ (cf. \cite[proof of Lemma 4.3]{RR2}). Then the argument used in the proof of Lemma \ref{L:1} shows that if $v$ corresponds to a point $x \in X(k)$, the map $H^i(\ell/k , T) \to H^i(L_{\bar{v}}/K_v , T)$ is injective. Consequently, if $X(k) \neq \emptyset$, the group $\Sha_0^i(L/K, T, V)$ is trivial.

Continuing with this case, suppose now that $k$ is a number field. As in the proof of Theorem \ref{T:1}, one can pick a subset $U \subset V^k_f$ with finite complement such that for any $u \in U$, there exists a point $x_u \in X$ whose residue field $k(x_u)$ is contained in $k_u$. Then one shows that $\Sha_0^i(L/K, T, V)$ is contained in
$$
\Sha_k^i(\ell/k, T, U) := \ker\left(H^i(\ell/k , T) \longrightarrow \prod_{u \in U} H^i(\ell_{\bar{u}}/k_u , T)  \right)
$$
Since $\Sha_k^i(\ell/k, T, U)$ is finite (for example, by \cite[Proposition 4.2]{RR2}; we note that it follows from the Poitou-Tate theorems that $H^i(\ell/k , T)$ is in fact finite for $i \geq 3$ --- see \cite[Ch. II, \S6]{Serre-GC}), we conclude that $\Sha_0^i(L/K, T, V)$ is finite as well. Thus, in the case where $\dim X = 1$, the group $\Sha^i(L/K , T, U)$ is finite if either $X(k) \neq \emptyset$ or $k$ is a number field, for all $i \geq 1$.

To conclude our discussion of this case, let us remark that one can construct a conic $X$ (= the Severi-Brauer variety for a quaternion algebra) over a finitely generated field $k$ such that for a suitable bi-quadratic extension $\ell/k$, the kernel of the map $H^3(\ell/k , T) \to H^3(L/K , T)$ is infinite for the 1-dimensional split torus $T = \mathbb{G}_m$. Thus, without the additional assumptions that we made in the preceding statements, the group $\Sha_0^3(L/K, T, V)$ may be infinite. This renders our approach inapplicable. To the best of our knowledge, the question about the finiteness of $\Sha^3(L/K, T, V)$ in this situation remains open.

\vskip2mm

Now suppose that $\dim X \geq 2.$ Recall that the proof of Lemma \ref{L:2} relied on the fact that if $Y \subset X$ is a geometrically integral smooth subvariety of codimension 1, then for a closed point $y \in Y(k)$ the map
$$
H^i(\ell(Y)/k(Y), T(\mathcal{O}_{Y, y})) \longrightarrow H^i(\ell(Y)/k(Y) , T)
$$
is injective for $i = 1$. But in fact this result also remains valid for $i = 2$ (see \cite[Theorem 4.3]{CT-S}), so the above argument yields the finiteness of $\Sha^2(L/K, T, V)$ if $k$ is finitely generated and $X(k) \neq \emptyset$ or if $k$ is a number field. On the other hand, as we have already noted, in the number field case, the groups $H^i(\ell/k , T)$ are finite whenever $i \geq 3$, and thus the groups $\Sha^i(L/K, T, V)$ are in fact finite for all $i \geq 1$ whenever $k$ is a number field.

\vskip2mm

\noindent {\small {\bf Acknowledgements.} We would like to thank Igor Dolgachev for useful correspondence concerning the Bertini theorems and the anonymous referee for valuable suggestions. The second-named author was partially supported by NSF grant DMS-2154408.}

\end{document}